\documentclass[a4paper,12pt]{amsart}

\usepackage[margin=1in]{geometry}

\usepackage{amsmath, amsthm, amssymb}
\usepackage{cite}
\usepackage[dvipdfmx]{graphicx}
\usepackage{bm}

\theoremstyle{definition}
\newtheorem{theorem}{Theorem}[section]
\newtheorem{lemma}[theorem]{Lemma}

\theoremstyle{definition}
\newtheorem{definition}[theorem]{Definition}

\newtheorem{pclaim}[theorem]{Claim}
\newtheorem*{ac}{Acknowledgments} 

\theoremstyle{remark}
\newtheorem{remark}[theorem]{Remark}

\newenvironment{rmenum}{
\begin{enumerate}

}
{\end{enumerate}}

\newenvironment{engenum}{
\begin{enumerate}

}
{\end{enumerate}}

\newcommand{\parNei}[2]{N_{#1}(#2)}

\newcommand{\tcomp}[2]{\mathcal{G}(#1, #2)}

\newcommand{\distgt}[4]{\lambda(#3, #4; #1, #2)}
\newcommand{\distgtf}[5]{\lambda(#4, #5; #3; #1, #2)}
\newcommand{\parcut}[2]{\delta_{#1}(#2)}

\newcommand{\conn}[1]{\mathcal{C}(#1)}
\newcommand{\conntodd}[2]{\mathcal{C}_\mathrm{odd}(#1; #2)}
\newcommand{\connteven}[2]{\mathcal{C}_\mathrm{even}(#1; #2)}
\newcommand{\layr}[2]{U_{ < #1}(#2)}
\newcommand{\layler}[2]{U_{ \le #1}(#2)}
\newcommand{\levelr}[2]{U_{#1}(#2)}
\newcommand{\init}[1]{K_{#1}}
\newcommand{\agtr}[3]{A_{(#1, #2)}(#3)}
\newcommand{\dgtr}[3]{D_{(#1, #2)}(#3)}
\newcommand{\cgtr}[3]{C_{(#1, #2)}(#3)}
\newcommand{\ar}[1]{A(#1)}
\newcommand{\dr}[1]{D(#1)}
\newcommand{\ccr}[1]{C(#1)}

\newcommand{\skgtx}[3]{(#1, #2)\langle #3 \rangle}

\newcommand{\extend}[5]{(#1, #2; #3) \oslash (#4, #5)}

\newcommand{\skcomb}[1]{\mathcal{F}(#1)} 
\newcommand{\primal}{\mathcal{P}}

\title[Bipartite Graft III]{Bipartite Graft III: General Case}
\author{Nanao Kita}
\address{Tokyo University of Science 2641 Yamazaki, Noda, Chiba, Japan 278-0022}
\email{kita@rs.tus.ac.jp}

\date{\today}

\begin{document}

\begin{abstract} 
This paper is a sequel of our previous paper (N. Kita: Bipartite graft {II}: Cathedral decomposition for combs. arXiv preprint arXiv:2101.06678, 2021).  
In our previous paper, 
a graft analogue of the Dulmage-Mendelsohn decomposition has been introduced for comb bipartite grafts.  
In this paper, 
we prove how this result can be extended for general bipartite grafts. 
\end{abstract}

\maketitle

\section{Definition} 

We use the notations and definitions used in Kita~\cite{kita2021bipartite}. 
In  the following, we present those that are newly introduced in this paper. 

\begin{definition} 
Let $G$ be a graph. 
We denote the set of connected component of $G$ by $\conn{G}$.  
Let $X_1, \ldots, X_k$, where $k \ge 1$,  be disjoint sets of vertices in $G$. We denote $G/X_1/\cdots /X_k$ by $G/\{ X_1, \ldots, X_k\}$. 
The vertex of $G/\{X_1, \ldots, X_k\}$ that corresponds to $X_i$ is denoted by $[X_i]$. 
\end{definition}

\begin{definition} 
Let $G$ be a graph, and let $T$ be a set of vertices from a supergraph of $G$. 
We denote by $\conntodd{G}{T}$ the set $\{ C\in \conn{G}: |V(C)\cap T| \mbox{ is odd }\}$ 
and by $\connteven{G}{T}$ the set $\{ C \in \conn{G}: |V(C)\cap T| \mbox{ is even }\}$. 
\end{definition}

\begin{definition} 
Let $(G, T)$ be a  graft, and let $r\in V(G)$. Let $F$ be a minimum join. 
We denote the set $\{ x\in V(G): \distgtf{G}{T}{F}{r}{x} = 0 \}$ by $\levelr{0}{r}$  
and the set $\{ x\in V(G): \distgtf{G}{T}{F}{r}{x} < 0 \}$ by $\layr{0}{r}$.  
We also denote the set $\levelr{0}{r} \cup \layr{0}{r}$ by $\layler{0}{r}$. 
\end{definition}

\section{Combic Sets} 

In this section, we define $F$-combic sets in grafts, where $F$ is a minimum join, 
and show a fundamental property to be used in later sections.

\begin{definition} 
Let $(G, T)$ be a graft, and let $F$ be a minimum join. 
We say that a set $X\subseteq V(G)$ is {\em $F$-combic} if it satisfies the following properties. 
\begin{rmenum} 
\item $E_G[X] \cap F = \emptyset$. 
\item For every $C \in \conntodd{ G - X}{T}$,  $ |\parcut{G}{C}\cap F | = 1$. 
\item For every $C \in \connteven{ G - X}{T}$,  $\parcut{G}{C}\cap F   = \emptyset$.  
\end{rmenum} 
If $X \subseteq V(G)$ is $F$-combic,  
let $H := G -  E_G[X]  - \bigcup \{ V(C): C \in \connteven{G-X}{T} \}$ and $S := T \cap V(H)$; 
 we call the bipartite graft $(H, S)/ \conntodd{G-X}{T}$ 
the {\em skeleton} of $X$, and denote this graft by $\skgtx{G}{T}{X}$.  
\end{definition} 

Note that $\skgtx{G}{T}{X}$ is a bipartite graft with color classes $X$ and $\{ [C] : C \in \conntodd{G-X}{T} \}$, 
and $F \cap E_G[X, \bigcup \{ V(C): C \in \conntodd{G- X}{T} \}]$ forms a join of $\skgtx{G}{T}{X}$.

\begin{lemma}  \label{lem:combic2min} 
Let $(G, T)$ be a graft, and let $F$ be a minimum join. 
If $X\subseteq V(G)$ is $F$-combic, then the following hold. 
\begin{rmenum} 
\item \label{item:combic2min:oc} 
Let $C \in \conntodd{ G - X}{T}$, 
let $e_C$ be the edge in $\parcut{G}{C}\cap F$, and let $r_C \in V(C)$ be the end of $e_C$.    
Then, $F\cap E(C)$ is a minimum join of the graft $(C, (T\cap V(C)) \Delta \{r_C\})$.   
\item  \label{item:combic2min:ec} 
For every $C \in \connteven{ G - X}{T}$,  $F\cap E(C)$ is a minimum join of the graft $(C, T\cap V(C))$.   
\item \label{item:combic2min:sk} 
$F \cap \bigcup \{ \parcut{G}{C} : C \in \conntodd{G-X}{T}  \}$ is a minimum join of $\skgtx{G}{T}{X}$. 
\end{rmenum} 
\end{lemma} 
\begin{proof} 
For proving \ref{item:combic2min:oc}, let $C \in \conntodd{ G - X}{T}$. 
Obviously, $F\cap E(C)$ is a join of $(C, (T\cap V(C)) \Delta \{r_C\})$.   
If $F\cap E(C)$ is not minimum, then Lemma~\ref{lem:minimumjoin} implies that $C$ has a circuit with a negative $F$-weight; 
however, this implies that $F$ is not a minimum join of $(G, T)$, contradicting the assumption. 
Hence, \ref{item:combic2min:oc} is proved. 
The claim \ref{item:combic2min:ec} can also be proved in a similar way. 

Let $F' := F \cap \bigcup \{ \parcut{G}{C} : C \in \conntodd{G-X}{T}  \}$.  
For proving \ref{item:combic2min:sk}, first note that $F'$ is obviously a join of $\skgtx{G}{T}{X}$. 
Let $Q$ be an arbitrary circuit of $\skgtx{G}{T}{X}$. 
Note $E(Q) \subseteq \bigcup \{ \parcut{G}{C} : C \in \conntodd{G-X}{T} \}$.  
For each $C \in \conntodd{G-X}{T}$, we have $w_F( Q. \parcut{G}{C}) \ge 0$. 
Hence, we have $w_F(Q) \ge 0$. 
That is, $\skgtx{G}{T}{X}$ has no circuit with negative $F'$-weight. 
Therefore, Lemma~\ref{lem:minimumjoin} implies that $F'$ is a minimum join. 
This completes the proof of the lemma. 
\end{proof}

\section{Primal Grafts}  

We introduce the concept of primal grafts and its fundamental property to be used in later sections.

\begin{definition} 
Let $(G, T)$ be a graft, and let $r\in V(G)$. Let $F$ be a minimum join of $(G, T)$. 
We say that $(G, T)$ is primal with respect to $x$ or $x$-primal if  $\distgtf{G}{T}{F}{r}{x} \le 0$ for every $x\in V(G)$. 
\end{definition}

\begin{lemma}[Seb\"o~\cite{DBLP:journals/jct/Sebo90}] \label{lem:sebo} 
Let $(G, T)$ be a graft, and let $F$ be a minimum join of $(G, T)$. 
Let $x, y \in V(G)$ be two distinct vertices, and let $P$ be an $F$-shortest path between $x$ and $y$. 
Then, $F\Delta E(P)$ is a minimum join of graft $(G, T\Delta \{x, y\})$.
Additionally, for every $z\in V(G)$, 
 $\distgtf{G}{T\Delta \{x, y\}}{F\Delta E(P)}{y}{z} = \distgtf{G}{T}{F}{x}{z} - w_F(P)$.  
\end{lemma}

Lemma~\ref{lem:sebo} implies the following for primal grafts.

\begin{lemma} \label{lem:tower2primal} 
Let $(G, T)$ be a graft that is primal with respect to $r\in V(G)$, and let $F$ be a minimum join of $(G, T)$. 
Let $r' \in \agtr{G}{T}{r} \setminus \{r\}$, and let $P$ be an $F$-shortest path between $r$ and $r'$. 
Then, $F \Delta E(P)$ is a minimum join of $(G, T\Delta \{r, r'\})$. 
Additionally, $\distgtf{G}{T}{F}{r}{x} = \distgtf{G}{T \Delta \{r, r'\}}{F\Delta E(P)}{r'}{x}$ for every $x\in V(G)$. 
Accordingly, the graft $(G, T \Delta \{r, r'\})$ is primal with respect to $r' \in V(G)$ 
for which $\agtr{G}{T\Delta \{r, r'\}}{r'} = \agtr{G}{T}{r}$ 
and $\dgtr{G}{T\Delta \{r, r'\}}{r'} = \dgtr{G}{T}{r}$ hold. 
\end{lemma}

\section{Seb\"o's Distance Theorem}

In this section, we introduce a classical theorem by Seb\"o~\cite{DBLP:journals/jct/Sebo90}  
regarding $F$-distances in grafts where $F$ is a minimum joins. 
This theorem is used in deriving our main results.

\begin{definition} 
Let $(G, T)$ be a graft, and let $x\in V(G)$. 
We call the member of $\conn{G[\layler{0}{x}]}$ with $x\in V(K)$ the {\em initial component} of $x$ and denote this component by $\init{x}$. 
We denote the set $V(\init{x})\cap \levelr{0}{x}$ by $\agtr{G}{T}{x}$,  
the set $V(\init{x})\setminus \agtr{G}{T}{x}$ by $\dgtr{G}{T}{x}$, 
and the set $V(G)\setminus \agtr{G}{T}{x}\setminus \dgtr{G}{T}{x}$ by $\cgtr{G}{T}{r}$. That is, $\cgtr{G}{T}{r} = V(G) \setminus V(\init{x})$.  
Regarding these notation, 
we often omit the subscript $(G, T)$ and write $\ar{r}$, $\dr{r}$, or $\ccr{r}$ if it is obvious from the context. 
\end{definition}

\begin{remark} 
A graft $(G, T)$ is primal with respect to $r\in V(G)$ if and only if $G$ is equal to the initial component of $r$. 
Hence,  for a primal graft $(G, T)$ with respect to $r$,  $\agtr{G}{T}{r}$ and $\dgtr{G}{T}{r}$ 
are equal to $\levelr{0}{r}$ and $\layr{0}{r}$, respectively. 
\end{remark}

The next theorem is a part of the main results from Seb\"o~\cite{DBLP:journals/jct/Sebo90}.

\begin{theorem}[Seb\"o~\cite{DBLP:journals/jct/Sebo90}]   \label{thm:sebo} 
Let $(G, T)$ be a bipartite graft, and let $r\in V(G)$. Let $F$ be a minimum join. 
Then, $\agtr{G}{T}{r}$ is an $F$-combic set that satisfies the following: 
\begin{rmenum} 
\item $\conntodd{ G - \agtr{G}{T}{r}}{T} = \conn{G[D(x)]}$ and $\connteven{G - \agtr{G}{T}{r}}{T} = \conn{G[C(x)]}$. 
\item $\skgtx{G}{T}{\agtr{G}{T}{r}}$ is an $r$-primal comb. 
\item For each $C \in \conntodd{ G - \agtr{G}{T}{r}}{T}$,  
let $r_K \in V(C)$ be the end of the sole edge from $\parcut{G}{C} \cap F$.  
Then, $(C,  ( T\cap V(C)) \Delta \{r_K\})$ is $r_K$-primal  
with $\min_{x \in V(C)} \distgtf{C}{( T\cap V(C)) \Delta \{r_K\}}{F\cap E(C)}{r}{x} 
= \min_{x\in V(G)} \distgtf{G}{T}{F}{r}{x} - 1$.  
\item For each $C \in \conntodd{ G - \agtr{G}{T}{r}}{T}$,  $\parNei{G}{\agtr{G}{T}{r}}\cap V(C) \subseteq \agtr{C}{( T\cap V(C)) \Delta \{r_K\}}{r_C}$. 
\end{rmenum} 
\end{theorem}

\section{Extreme Sets} 

We define extreme and bipartitic extreme sets in grafts or bipartite grafts  
and introduce a notation typically used for extreme sets. 

\begin{definition} 
Let $(G, T)$ be a graft, and let $F$ be a minimum join. 
We say that a set $X\subseteq V(G)$ is {\em extreme} if 
$\distgtf{G}{T}{F}{x}{y} \ge 0$ for every $x, y\in V(G)$. 
\end{definition} 

\begin{definition} 
Let $(G, T)$ be a bipartite graft with color classes $A$ and $B$, and let $F$ be a minimum join. 
We say that an extreme set $X\subseteq V(G)$ is {\em bipartitic} if $X\subseteq A$ or $X \subseteq B$ holds. 
\end{definition} 

The following notation is used  in later sections where $X$ is typically an extreme set. 

\begin{definition} 
Let $(G, T)$ be a graft, let $F$ be a minimum join, and let $X\subseteq V(G)$. 
We denote by $D_X$ the set of vertices 
$\{ y \in V(G)\setminus X : \min_{x \in X} \distgtf{G}{T}{F}{x}{y} < 0\}$.  
We denote the set $V(G)\setminus X \setminus D_X$ by $C_X$. 
\end{definition}

\section{Trivial Vertices and Maximal Bipartitic Extreme Sets} 

In this section, we define the concept of trivial vertices in a graft  
and show some properties of trivial vertices and extreme sets. 

\begin{definition} 
A factor-component $C\in \tcomp{G}{T}$ with $|V(C)| = 1$ is said to be {\em trivial}. 
The sole vertex of the trivial factor-component is also said to be {\em trivial} in the graft. 
Trivial vertices are disjoint from $T$, and 
no trivial vertex is connected to an edge from  minimum joins. 
\end{definition}

\begin{lemma}  \label{lem:extreme2trivial} 
Let $(G, T)$ be a bipartite graft with color classes $A$ and $B$. 
Let $X\subseteq A$ be a maximal bipartitic extreme set.  
Then, 
\begin{rmenum} 
\item \label{item:extreme2trivial:noedge} $E_G[D_X, C_X] = \emptyset$; and, 
\item \label{item:extreme2trivial:c2trivial} every vertex from $C_X$ is a trivial vertex from $B$. 
\end{rmenum} 
\end{lemma} 
\begin{proof} 
Let $F$ be a minimum join of $(G, T)$. 
Suppose that $y\in D_X$ and $z \in C_X$ are adjacent.  
Then, $| \distgtf{G}{T}{F}{x}{z} - \distgtf{G}{T}{F}{x}{y} | = 1$ for every $x\in X$. 
This implies $\min_{x\in X} \distgtf{G}{T}{F}{x}{z} \le 0$. 
This contradicts the definition of $C_X$ or the maximality of $X$. 
Hence, \ref{item:extreme2trivial:noedge} follows. 

For proving \ref{item:extreme2trivial:c2trivial},  
first note that every vertex in $C_X$ are in $B$; otherwise, it would contradict the maximality of $X$. 
Next, note that $\parcut{G}{X\cup D_X}\cap F = \emptyset$ holds for every minimum join $F$, due to the definition of $D_X$. 
Now, suppose that $C_X$ contains a nontrivial vertex $v$. 
Let $G_v  \in \tcomp{G}{T}$ be the factor-component with $v\in V(G_v)$. 
Then, $v$ is adjacent to a vertex $w \in  A \cap V(G_v)$. 
However, because $V(G_v) \subseteq C_X$ holds, $w \in A\cap C_X$ follows, which is a contradiction.  
Thus, \ref{item:extreme2trivial:c2trivial} is proved.  
\end{proof}

\begin{lemma} \label{lem:path2trivial} 
Let $(G, T)$ be a graft, and let $F$ be a minimum join. 
Let $X \subseteq V(G)$ an extreme set, 
and let $Y, Z\subseteq V(G)$ be sets with $Y \dot\cup Z = V(G)\setminus X$, $E_G[Y, Z] = \emptyset$, 
and $(E_G[Z] \cup \parcut{G}{Z}) \cap F = \emptyset$.  
Let $(\hat{G}, \hat{T})$ be a graft 
such that 
$V(\hat{G}) \supseteq V(G)$, 
$E(\hat{G}) \supseteq E(G)$, $E(\hat{G})\setminus E(G) \subseteq \{uv: u\in X, v\in  V(\hat{G})\setminus V(G) \}$, 
and $\hat{T} = T$. 
Let $\hat{Z} := Z \cup ( V(\hat{G}) \setminus V(G) )$. 

Let $Q$ be a subgraph of $\hat{G}$ with  $V(Q) \cap \hat{Z} \neq \emptyset$. 
\begin{rmenum} 
\item \label{item:path2trivial:c} If $Q$ is a circuit or a path between a vertex in $X$ and a vertex in $X \cup Z$, 
then $w_F(Q) > 0$. 
\item \label{item:path2trivial:d} 
If $Q$ is a path between a vertex $y$ in $Y$ and a vertex in $X$, 
then $w_F(Q) > \min_{x \in X} \distgtf{G}{T}{F}{x}{y}$. 
\end{rmenum} 
\end{lemma} 
\begin{proof} 
In each case, $Q$ is the sum of two subgraphs $Q. E_{\hat{G}}[X \cup Y]$ and $Q. E_{\hat{G}}[Z] \cup \parcut{\hat{G}}{Z}$. 
These subgraphs consist of disjoint paths  that have at least one edge and fit one of the following types: 
\begin{engenum} 
\item \label{item:y}  the ends are in $X$, and the edges are in $E_{\hat{G}}[X \cup Y]$, 
\item \label{item:z}  the ends are in $X$, and  the edges are in $E_{\hat{G}}[Z] \cup \parcut{\hat{G}}{Z}$, or 
\item \label{item:end} one end is in $X$, whereas the other end is in $Y\cup \hat{Z}$. 
\end{engenum}

If $Q$ is a circuit or a path between vertices in $X$, then $Q$ is a sum of paths of type~\ref{item:y} or \ref{item:z}.  
If $Q$ is a path between a vertex in $X$ and a vertex in $Y$, 
then $Q$ is a sum of paths of type of \ref{item:y} or \ref{item:z} and  
 one path of type~\ref{item:end} that has all its edges in $E_{\hat{G}}[X \cup Y]$. 
If $Q$ is a path between a vertex in $X$ and a vertex in $\hat{Z}$, 
then $Q$ is a sum of paths of type~\ref{item:y} or \ref{item:z} and one path of type~\ref{item:end}
that  has all its edges in $E_{\hat{G}}[Z] \cup \parcut{\hat{G}}{Z}$.  

For every case, at least one path is of type~\ref{item:z} or type~\ref{item:end} with positive weight.  
Every path of type~\ref{item:y} or \ref{item:z} has an $F$-weight no less than $0$ or $2$, respectively. 
The path of type~\ref{item:end} has an $F$-weight no less than $\min_{x\in X} \distgtf{G}{T}{F}{x}{y}$ or $1$ 
for the cases $y\in Y$ or $y\in \hat{Z}$, respectively. 
Therefore, \ref{item:path2trivial:c} and \ref{item:path2trivial:d} follow. 
\end{proof}

\begin{lemma} \label{lem:tr2del} 
Let $(G, T)$ be a bipartite graft with color classes $A$ and $B$. 
Let $F$ be a minimum join of $(G, T)$. 
Let $X\subseteq A$ be a maximal bipartitic extreme set.  
Let $C' \subseteq C_X$, and let $(G', T')$ be a graft such that $G' = G - C'$ and $T' = T \setminus C'$. 
Then, the following properties hold. 
\begin{rmenum} 
\item \label{item:tr2del:join} A set of edges is a minimum join of $(G', T')$ if and only if it is a minimum join of $(G, T)$. 
\item \label{item:tr2del:d} For every $y\in D_X$, 
 $\min_{x\in X} \distgtf{G'}{T'}{F}{x}{y} = \min_{x\in X} \distgtf{G}{T}{F}{x}{y}$.  
\item \label{item:tr2del:c} For every $x\in X$ and every $y\in C_X\setminus C'$, $\distgtf{G'}{T'}{F}{x}{y} > 0$. 
\item \label{item:tr2del:x} For every $x, y\in X$, $\distgtf{G'}{T'}{F}{x}{y} \ge 0$. 
\item \label{item:tr2del:ext} $X$ is maximal bipartitic extreme in  $(G', T')$.  
\end{rmenum} 
\end{lemma} 
\begin{proof} 
Every join of $(G', T')$ is obviously a join of $(G, T)$. 
Conversely, Lemma~\ref{lem:extreme2trivial} \ref{item:extreme2trivial:c2trivial} implies that a join of $(G, T)$ is also a join of $(G', T')$. 
That is, a set of edges is a join of $(G', T')$ if and only if it is a join of $(G, T)$. 
Hence, \ref{item:tr2del:join} is proved.

Therefore, $F$ is a minimum join of $(G', T')$. 
Lemma~\ref{lem:path2trivial} \ref{item:path2trivial:d} easily implies \ref{item:tr2del:d}. 
The statements  \ref{item:tr2del:c} and \ref{item:tr2del:x} obviously follow.  
Then, \ref{item:tr2del:d}, \ref{item:tr2del:c}, and \ref{item:tr2del:x} implies \ref{item:tr2del:ext}. 
This completes the proof. 
\end{proof}

\begin{lemma} \label{lem:tr2add} 
Let $(G, T)$ be a bipartite graft with color classes $A$ and $B$. 
Let $F$ be a minimum join of $(G, T)$. 
Let $X\subseteq A$ be a maximal bipartitic extreme set.  
Let $(\hat{G}, \hat{T})$ be a graft 
such that 
$V(\hat{G}) \supseteq V(G)$, 
$E(\hat{G}) \supseteq E(G)$, $E(\hat{G})\setminus E(G) \subseteq \{uv: u\in X, v\in  V(\hat{G})\setminus V(G) \}$, 
and $\hat{T} = T$. 
Let $\hat{C}_X := C_X  \cup ( V(\hat{G})) \setminus V(G))$. 
Then, the following properties hold. 
\begin{rmenum} 
\item \label{item:tr2add:join} A set of edges is a minimum join of $(\hat{G}, \hat{T})$ if and only if it is a minimum join of $(G, T)$. 
\item \label{item:tr2add:d} For every $y\in D_X$, 
 $\min_{x\in X} \distgtf{\hat{G}}{\hat{T}}{F}{x}{y} = \min_{x\in X} \distgtf{G}{T}{F}{x}{y}$.  
\item \label{item:tr2add:c} For every $x\in X$ and every $y\in \hat{C}_X$, $\distgtf{\hat{G}}{\hat{T}}{F}{x}{y} > 0$. 
\item \label{item:tr2add:x}  For every $x, y\in X$, $\distgtf{\hat{G}}{\hat{T}}{F}{x}{y} \ge 0$. 
\item \label{item:tr2add:ext} $X$ is maximal bipartitic extreme in  $(\hat{G}, \hat{T})$.  
\end{rmenum} 
\end{lemma} 
\begin{proof} 
First, $F$ is obviously a join of $(\hat{G}, \hat{T})$. 
Lemmas~\ref{lem:path2trivial}  \ref{item:path2trivial:c} and \ref{lem:minimumjoin}  further imply 
that $F$ is a minimum join of $(\hat{G}, \hat{T})$. 
Therefore, 
we can apply Lemma~\ref{lem:tr2del} to $(\hat{G}, \hat{T})$ and $(G, T)$, 
which proves \ref{item:tr2add:d}, \ref{item:tr2add:c}, and \ref{item:tr2add:x}. 
The statements \ref{item:tr2add:d}, \ref{item:tr2add:c}, and \ref{item:tr2add:x} imply \ref{item:tr2add:ext}.  
Then, \ref{item:tr2add:ext} and Lemma~\ref{lem:extreme2trivial} \ref{item:extreme2trivial:c2trivial} 
imply that every join of $(\hat{G}, \hat{T})$ is a join of $(G, T)$; the converse clearly holds. 
Thus, \ref{item:tr2add:join} is proved. 
This completes the proof.

\end{proof}

Let $(G, T)$ be a bipartite graft with color classes $A$ and $B$. 
Let $X\subseteq A$ be a maximal bipartitic extreme set.  
Under Lemma~\ref{lem:extreme2trivial}, we call the set of vertices in $C_X$ the {\em fringe} of $X$. 
We call the operation of deleting from $(G, T)$  all vertices in $C_X$ the {\em fringe removal}. 
We call the operation of adding some new vertices and edges between new vertices and vertices in $X$ 
the {\em fringe addition}. 
Lemmas~\ref{lem:tr2del} and \ref{lem:tr2add} 
imply that fringe deletion and addition do not make any essential changes to the minimum join structure,  
which can be summarized as follows.

\begin{theorem} \label{thm:fr} 
Let $(G, T)$ be a bipartite graft with color classes $A$ and $B$, 
and let $F$ be a minimum join of $(G, T)$. 
Let $X\subseteq A$ be a maximal bipartitic extreme set of $(G, T)$.  
Let $(G', T')$ be a graft obtained by fringe deletion or addition. 
Then, the following holds. 
\begin{rmenum} 
\item \label{itme:fr:join} A set of edges is a minimum join of $(G', T')$ if and only if it is a minimum join of $(G, T)$. 
\item \label{item:fr:d} For every $y\in D_X$, 
 $\min_{x\in X} \distgtf{G}{T}{F}{x}{y} = \min_{x\in X} \distgtf{G'}{T'}{F}{x}{y}$.  
\item \label{item:fr:ext} $X$ is maximal bipartitic extreme in  $(G', T')$.   
\end{rmenum} 
\end{theorem}

\section{Reduction between Combs and General Bipartite Grafts}  \label{sec:reduction} 
\subsection{Decomposition}  \label{sec:reduction:decompose} 

In this section, we show that a general bipartite graft can be considered 
as a combination of comb and primal bipartite grafts.  
We prove that, for graft $(G, T)$ and minimum join $F$,  every maximal bipartitic extreme set $X$ is an $F$-combic set 
for which each member of $\conntodd{G-X}{T}$ induces a primal graft. 
We define a new operation called rootlization that   
adds some vertices and edges to $(G, T)$ and 
makes the entire graft, except some trivial parts, primal with respect to a new vertex. 
We then derive the desired property by applying Theorem~\ref{thm:sebo} to the rootlized graft.

\begin{definition} 
Let $(G, T)$ be a graft, and let $X \subseteq V(G)$. 
Let $\hat{G}$ be the graph such that $V(\hat{G}) = V(G) \cup \{r, s\}$, where $r, s\not\in V(G)$,  
and $E(\hat{G}) = E(G) \cup \{rs \} \cup \{ sx : x\in X\}$. 
Let $\hat{T} = T \cup \{r, s\}$. 
We call  the graft $(\hat{G}, \hat{T})$  the {\em rootlization} of $(G, T)$ 
by the {\em mount} $X$, the {\em root} $r$, and the {\em attachment} $s$. 
We also denote $(\hat{G}, \hat{T})$ by $\extend{G}{T}{X}{r}{s}$. 
\end{definition}

The next lemma shows a fundamental property of rootlization 
and is used for proving Lemma~\ref{lem:maxext2init} and Theorem~\ref{thm:decompose}.

\begin{lemma} \label{lem:extend} 
Let $(G, T)$ be a graft, and let $X \subseteq V(G)$. 
Let $(\hat{G}, \hat{T})$ be the rootlization of $(G, T)$ by the mount $X$, the root $r$, and the attachment $s$. 
\begin{rmenum} 
\item \label{item:extend:join} A set $\hat{F} \subseteq E(\hat{G})$ is a minimum join if and only if 
$\hat{F}$ is of the form $\hat{F} = F \cup \{rs\}$, where $F$ is a minimum join of $(G, T)$. 
\item \label{item:extend:distance} Let $F$ be a minimum join of $(G, T)$, and let  $\hat{F} = F \cup \{rs\}$. 
Then, $\distgtf{\hat{G}}{\hat{T}}{\hat{F}}{r}{y} = \min_{x\in X} \distgtf{G}{T}{F}{x}{y}$ holds for every $y\in V(G)$, 
whereas $\distgtf{\hat{G}}{\hat{T}}{\hat{F}}{r}{s} = -1$. 
\end{rmenum} 
\end{lemma} 
\begin{proof}  
Let $\hat{F} := F \cup \{ rs \}$, where $F$ is a minimum join of $(G, T)$.  
Obviously, $\hat{F}$ is a join of $(\hat{G}, \hat{T})$. 

\begin{pclaim} \label{claim:extend:circuit} 
No circuit in $\hat{G}$ contains $\{ rs \}$. 
No circuit in $\hat{G}$ with negative $\hat{F}$-weight contains edges from $E_{\hat{G}}[s, X]$. 
\end{pclaim} 
\begin{proof} 
The first statement obviously holds. 
For proving the second one, let $C$ be a circuit with $E(C) \cap E_{\hat{G}}[s, X] \neq \emptyset$. 
The first statement implies  $E(C) \subseteq E_{\hat{G}}[s, X] \cup E(G)$.  
Obviously, $w_{\hat{F}}( C. E_{\hat{G}}[s, X] ) > 0$.  
The graph $C. E(G)$ is a path between vertices in $X$. 
Because $X$ is extreme, $w_{F}(C. E(G)) \ge 0$. 
Hence, $w_{\hat{F}}(C) \ge 0$ follows, and the second statement is proved. 
\end{proof} 

We now prove the necessity of \ref{item:extend:join} by proving that $\hat{F}$ is a minimum join. 
Lemma~\ref{lem:minimumjoin} implies that $G$ has no circuit with negative $F$-weight. 
Claim~\ref{claim:extend:circuit} further implies that $\hat{G}$ has no circuit with negative $\hat{F}$-weight.  
Hence, by Lemma~\ref{lem:minimumjoin} again, $\hat{F}$ is a minimum join of $(\hat{G}, \hat{T})$. 
This proves the necessity of \ref{item:extend:join}.

We next prove the sufficiency of \ref{item:extend:join}. 
Suppose that the claim fails.  
Then, there is a minimum join $\hat{F}'$ that contains an edge from $E_{\hat{G}}[s, X]$. 
The subgraph of $G$ determined by $\hat{F} \Delta \hat{F}'$  contains a circuit $Q$ 
 with $\hat{F}$-weight $0$ that contains an edge from $E_{\hat{G}}[s, X]$. 
 This contradicts Claim~\ref{claim:extend:circuit}. 
Therefore, the sufficiency of \ref{item:extend:join} is proved. 

The statement \ref{item:extend:distance} can easily be confirmed  from \ref{item:extend:join}.

\end{proof}

Lemma~\ref{lem:extend} implies the next lemma, which connects rootlized graphs and Theorem~\ref{thm:sebo}.

\begin{lemma}  \label{lem:maxext2init}
Let $(G, T)$ be a bipartite graft with color classes $A$ and $B$, and let  $F$ be a minimum join of $G$. 
Let $X \subseteq V(G)$ be a maximal bipartitic extreme set. 
Let $(\hat{G}, \hat{T}):= \extend{G}{T}{X}{r}{s}$.   
Let $\hat{F} := F \cup \{rs\}$, $\hat{X} := X \cup \{r\}$,  and let $\hat{D}_X:= D_X \cup \{s\}$.  
Then, $\hat{G}[\hat{X} \cup \hat{D}_X]$ is the initial component of $r$, for which  
 $\agtr{\hat{G}}{\hat{T}}{r} = \hat{X}$ and $\dgtr{\hat{G}}{\hat{T}}{r} = \hat{D}_X$. 
\end{lemma} 
\begin{proof} 
Lemma~\ref{lem:extend} \ref{item:extend:distance} implies $\hat{X} = \levelr{0}{r}$ and $\hat{D}_X = \layr{0}{r}$. 
Additionally, $\hat{G}[\hat{X}\cup \hat{D}_X]$ is obviously connected. 
Accordingly, the claim follows. 
\end{proof}

From Theorem~\ref{thm:sebo} and Lemmas~\ref{lem:extend} and \ref{lem:maxext2init},  
we obtain Theorem~\ref{thm:decompose}.

\begin{theorem} \label{thm:decompose} 
Let $(G, T)$ be a bipartite graft with color classes $A$ and $B$.  Let $F$ be a minimum join. 
Let $X \subseteq A$ be a maximal bipartitic extreme set. 
Then,  $X$ is an $F$-combic with the following properties:  
\begin{rmenum} 
\item \label{item:decompose:odd} $\conn{ G[ D_X ] }  = \conntodd{G - X}{T} $. 
\item \label{item:decompose:comp} For every $C \in \conntodd{G - X}{T}$ and every $r_C \in V(C)$ 
that is the end of the edge from $\parcut{G}{C}\cap F$,   
the graft $(C, ( T\cap V(C) ) \Delta \{r_C\})$ is primal with respect to $r_C$, and 
$\parNei{G}{X}\cap V(C) \subseteq \agtr{C}{( T\cap V(C) ) \Delta \{r_C\}}{r_C}$ holds. 
\item \label{item:decompose:even} Every member of $\connteven{G - X}{T}$ is trivial factor-component. 
\end{rmenum} 
\end{theorem}  
\begin{proof} 
Let $(\hat{G}, \hat{T}) := \extend{G}{T}{X}{r}{s}$.   
Let $\hat{F} := F \cup \{rs\}$; Lemma~\ref{lem:extend} implies that $\hat{F}$ is a minimum join of $(\hat{G}, \hat{T})$.   
Let $\hat{X} := X \cup \{r\}$, and let $\hat{D}_X := D_X \cup \{s\}$. Let $\hat{C}_X := V(\hat{G})\setminus \hat{X} \setminus \hat{D}_X = C_X$.  
According to Lemma~\ref{lem:maxext2init}, 
the initial component of $r$ is $\hat{G}[\hat{X} \cup \hat{D}_X]$
for which $\hat{X} = \ar{r}$ and $\hat{D}_X = \dr{r}$.

Hence, applying Theorem~\ref{thm:sebo} to $(\hat{G}, \hat{T})$ with root $r$,   
we obtain that $\hat{X}$ is a $\hat{F}$-combic set that satisfies the following: 
\begin{engenum} 
\item \label{item:d2odd} $\conntodd{\hat{G} - \hat{X} }{\hat{T}} =  \conn{ \hat{G}[ \hat{D}_X ] }$. 
\item \label{item:comp2nonpositive} For every $K \in \conntodd{\hat{G} - \hat{X} }{\hat{T}}$ 
and every vertex $r_K \in V(K)$ that is the end of the edge from $\parcut{\hat{G}}{K} \cap \hat{F}$,  
the graft $(K, (V(K) \cap T) \Delta \{r_K\})$ is primal with respect to $r_K$, and 
$\parNei{\hat{G}}{\hat{X}} \cap V(K)\subseteq \agtr{K}{(V(K) \cap T) \Delta \{r_K\}}{r_K}$ holds.

\end{engenum}

It is obvious that $\conntodd{\hat{G}-\hat{X}}{\hat{T}} = \conntodd{G- X}{T} \cup \{ G[s]\}$ and $\conn{\hat{G}[ \hat{D}_X] } = \conn{ G[D_X] } \cup \{ G[s] \}$ hold. 
Hence, \ref{item:d2odd} and \ref{item:comp2nonpositive} prove \ref{item:decompose:odd} and \ref{item:decompose:comp}.

From \ref{item:decompose:odd}, we have $\connteven{G-X}{T} = \conn{G[ C_X]}$.   
Therefore, from Lemma~\ref{lem:extreme2trivial} \ref{item:extreme2trivial:c2trivial}, we obtain \ref{item:decompose:even}. 
This completes the proof. 
\end{proof}

\subsection{Construction} \label{sec:reduction:construct} 

In this section, we show the converse of Theorem~\ref{thm:decompose}. 
We define an operation called synthesis, 
which construct a bipartite graft from given comb and set of primal bipartite grafts, 
and prove that the spine set of the  comb  
forms a combic set in the obtained bipartite graft.

\begin{definition} 
Let $(G, T)$ be a comb with spine set $A$ and tooth set $B$, 
and let $\{ (G_v, T_v) : v\in B\}$ be a set of mutually disjoint grafts that are also disjoint from $(G, T)$.   
For each $v \in B$,  
 $(G_v, T_v)$ is a bipartite graft with color classes $A_v$ and $B_v$ that is primal with respect to $r_v \in A_v$. 
Let $(\hat{G}, \hat{T})$ be a bipartite graft that satisfies the following: 
\begin{rmenum}  
\item For every $v\in B$, $G_v$ is a subgraph of $\hat{G}$ with $E_{\hat{G}}[G_v] = E(G_v)$ and $\hat{T}\cap V(G_v) = T_v \Delta \{r_v\}$; 
\item $(G, T) = (\hat{G}, \hat{T})/ \{ G_v : v\in B \}$, where each $G_v$ is contracted to $v$; and, 
\item $\parcut{\hat{G}}{G_v} = E_{\hat{G}}[A, \agtr{G_v}{T_v}{r_v}]$. 
\end{rmenum}

We call $(\hat{G}, \hat{T})$ a {\em synthesis} of $(G, T; A, B)$ and $\{(G_v, T_v; r_v) : v\in B\}$.  
Here $(G, T; A, B)$ is called a {\em skeleton comb},   
and $(G_v, T_v)$ is called a {\em tooth graft} associated with $v$ for each $v\in B$.  
\end{definition} 

\begin{remark} 
Note that the synthesis $(\hat{G}, \hat{T})$ is a bipartite graft 
with color classes $A \cup \bigcup_{v \in B} B_v$ and $B \cup \bigcup_{v\in B} A_v$.  
We always assume these to be two color classes of a synthesis unless stated otherwise. 
\end{remark}

The next lemma is used for proving Theorem~\ref{thm:construct}.

\begin{lemma} \label{lem:nonregpath}  
Let $(G, T)$ be a bipartite graft that is primal with respect to $r\in V(G)$. 
Then, $\agtr{G}{T}{r}$ is an extreme set. 
\end{lemma} 
\begin{proof} 
Let $F$ be a minimum join of $(G, T)$.  
Let $l := \min_{ x \in V(G) } \distgtf{G}{T}{F}{r}{x} $. 
If $l = 0$, then $V(G) = \{r\}$, and the statement trivially holds. 
We prove the case $l < 0$ by induction on $l$. 

Next, let $l \le -1$, and assume that the statement holds for every case where $l$ is smaller.  
Let $x, y \in \agtr{G}{T}{r}$, and let $Q$ be an $F$-shortest path between $x$ and $y$. 
Because $\agtr{G}{T}{r}$ is stable, we have $E(Q) \subseteq \bigcup \{ \parcut{G}{K} \cup E(K) : K \in \conn{ G[ \dgtr{G}{T}{r} ]} \}$.

Let $K\in \conn{ G[\dgtr{G}{T}{r}] }$. 
The set $E(Q)\cap \parcut{G}{K}$ has an even number of edges, 
only one of which can be from $F$, according to Theorem~\ref{thm:sebo}. 
Hence, we have $w_F( Q. \parcut{G}{K} ) \ge 0$.

\begin{pclaim}  \label{claim:nonregpath:seg} 
Let $u, v \in V(K) \cap \parNei{G}{ \agtr{G}{T}{r} }$.  If $P$ is a path of $K$ between $u$ and $v$, 
then $w_F(P) \ge 0$ holds. 
\end{pclaim} 
\begin{proof} 
Now, let $r_K \in V(K)$ be the end of the edge from $\parcut{G}{K} \cap F$.  
According to Theorem~\ref{thm:sebo},  
$F \cap E(K)$ is a minimum join of the graft $(K, (T\cap V(K))\Delta \{ r_K \})$, 
  the graft $(K, (T\cap V(K))\Delta \{ r_K \})$ is primal with respect to $r_K$, 
and $\min_{x\in V(K)} \distgtf{K}{(T\cap V(K))\Delta \{ r_K \}}{F\cap E(K)}{r_K}{x} = l+1$.  
Hence, the induction hypothesis implies 
that $\agtr{K}{(T\cap V(K))\Delta \{ r_K \}}{r_K}$ is an extreme set of $(K, (T\cap V(K))\Delta \{ r_K \})$. 
In addition, 
Theorem~\ref{thm:sebo} also implies  $V(K)\cap \parNei{G}{\agtr{G}{T}{r}} \subseteq \agtr{K}{(T\cap V(K))\Delta \{ r_K \}}{r_K}$. 
This proves the claim. 
\end{proof}

Claim~\ref{claim:nonregpath:seg} implies that 
each connected component of $Q. E(K)$ is a path whose $F$-weight is no less than $0$. 
Hence, we have $w_F(Q) \ge 0$. 
This proves the lemma. 
\end{proof}

We now prove Theorem~\ref{thm:construct}.

\begin{theorem} \label{thm:construct} 
Let $(\hat{G}, \hat{T})$ be a synthesis of a skeleton comb $(G, T; A, B)$ and a set of tooth grafts $\{ (G_v, T_v; r_v): v\in B\}$.  
Let $F$ be a minimum join of $(G, T)$.   
For each $v\in B$, let $f_v\in E(\hat{G})$ be the edge of $\hat{G}$ that corresponds to the edge from $\parcut{G}{v} \cap F$, 
and let $r_v' \in V(G_v)$ be the end of $f_v$. 
Let $F_v$ be a minimum join of $(G_v, T_v)$ if $r_v' = r_v$; 
otherwise, let $F_v$ be a minimum join of $(G_v, T_v \Delta \{r_v, r_v'\})$.     

Then, 
\begin{rmenum} 
\item  $F \cup \bigcup_{v\in B} F_v$ is a minimum join of $(\hat{G}, \hat{T})$, and
\item  $A$ is a maximal bipartitic extreme set of $(\hat{G}, \hat{T})$ with $C_A = \emptyset$.  
\end{rmenum} 

\end{theorem} 
\begin{proof}  
First, let $\hat{F} := F \cup \bigcup_{v\in B} F_v$. 
\begin{pclaim} \label{claim:construct:nonneg} 
Let $Q$ be a subgraph of $\hat{G}$ that is a circuit or a path between vertices in $A$. 
Then, $w_{\hat{F}}(Q) \ge 0$. 
\end{pclaim} 
\begin{proof} 
Note $E(Q) \subseteq \bigcup_{v\in B} \parcut{\hat{G}}{G_v}  \cup E(G_v)$.   
For every $v\in B$, $Q. \parcut{\hat{G}}{G_v}$ has an even number of edges, among which only one can be from $\hat{F}$. 
Hence, $Q. \parcut{\hat{G}}{G_v}$ has a $\hat{F}$-weight that is no less than $0$.  
If $Q. E(G_v)$ is empty, then its $\hat{F}$-weight is obviously $0$. 
Assume otherwise. 
Lemma~\ref{lem:minimumjoin} implies that $Q. E(G_v)$ cannot be a circuit for any $v\in B$.  
Consequently, 
each connected component of $Q. E(G_v)$ is a path between vertices in $A_v$; 
hence, Lemmas~\ref{lem:nonregpath} and \ref{lem:tower2primal} implies that it has a $\hat{F}$-weight no less than $0$. 
Thus, the claim follows. 
\end{proof}

We next prove the first statement.  
Obviously, $\hat{F}$ is a join of $(\hat{G}, \hat{T})$.  
Claim~\ref{claim:construct:nonneg} and Lemma~\ref{lem:minimumjoin} further prove that $\hat{F}$ is a minimum join.

We next prove the remaining statement.  
Obviously, $A$ is contained in a single color class of $\hat{G}$. 
Claim~\ref{claim:construct:nonneg} implies that $A$ is extreme. 
Let $x\in V(\hat{G})\setminus A$, and let $v\in B$ be the index with $x\in V(G_v)$. 
Let $P$ be an $F_v$-shortest path between $r_v'$ and $x$. 
Lemma~\ref{lem:tower2primal} ensures that $(G_v, T_v\Delta\{r_v, r_v'\})$ is also primal with respect to $r_v'$ for the case $r_v\neq r_v'$. 
Hence, the $F_v$-weight of $P$ is no greater than $0$. 
Thus, $P + f_v$ is a path between $x$ and a vertex in $A$ with negative $\hat{F}$-weight. 
Hence, $A \cup D_A = V(\hat{G})$, which implies that $A$ is maximal bipartitic extreme.  
This completes the proof of the theorem. 
\end{proof}

\subsection{Characterization}

In this section, we summarize the main results from Sections~\ref{sec:reduction:decompose} and \ref{sec:reduction:construct}.

Theorems~\ref{thm:fr}, \ref{thm:decompose} and \ref{thm:construct} imply the following theorem.

\begin{theorem} 
Let $(\hat{G}, \hat{T})$ be a bipartite graft. 
If  $(\hat{G}, \hat{T})$ is a synthesis of a skeleton comb $(G, T; A, B)$ and a set of tooth grafts $\{ (G_v, T_v; r_v): v\in B\}$,  
then $A$ is a maximal bipartitic extreme set of $(\hat{G}, \hat{T})$ with empty fringe. 
Conversely, 
if $A$ is a maximal bipartitic extreme set of $(\hat{G}, \hat{T})$, 
then the graft obtained from $(\hat{G}, \hat{T})$ by the fringe removal is a synthesis of skeleton comb whose spine set is $A$ and a set of tooth grafts. 
\end{theorem}

Lemma~\ref{lem:combic2min} and Theorems~\ref{thm:decompose} and \ref{thm:construct} also imply the following theorem. 

\begin{theorem} 
Let $(\hat{G}, \hat{T})$ be a synthesis of a skeleton comb $(G, T; A, B)$ and a set of tooth grafts $\{ (G_v, T_v; r_v): v\in B\}$.  
Then, 
\begin{rmenum} 
\item  a set of edges is a minimum join of $(\hat{G}, \hat{T})$  
if and only if it is of the form $F \cup \bigcup_{v\in B} F_v$, 
where $F$ is a minimum join of $(G, T; A, B)$, 
$F_v$ is a minimum join of $(G_v, T_v\Delta \{r_v'\})$  for every $v\in B$,  
and $r_v'\in V(G_v)$ is the end of the edge of $G$ that corresponds to the edge from $\parcut{G}{v}\cap F$. 
\item 
For every minimum join of $(G, T; A, B)$, 
there is a minimum join of $(\hat{G}, \hat{T})$ that contains it. 
\end{rmenum} 
\end{theorem}

\section{Characterization of Primal Grafts} \label{sec:charprimal} 

Section~\ref{sec:reduction} reveals that 
every bipartite graft can be characterized as a synthesis of combs and primal bipartite grafts. 
In this section, we characterize primal bipartite grafts in terms of synthesis. 
That is, every primal bipartite graft is a synthesis in which the skeleton comb is primal.

\begin{theorem} \label{thm:primal2char} 
A bipartite graft $(G, T)$ is primal with respect to $r \in V(G)$ if and only if 
it is a synthesis of  
a skeleton comb $(G', T'; A, B)$ that is primal with respect to $r \in A$
 and a set of tooth grafts.  
\end{theorem} 
\begin{proof} 
The sufficiency is immediate from Theorem~\ref{thm:sebo}. 
We prove the necessity in the following. 
Let $(G, T)$ be a synthesis of a skeleton comb $(G', T'; A, B)$ that is primal with respect to $r\in A$ 
and a set of tooth grafts $\{ (G_v, T_v; p_v): v\in B\}$.  
Let $F$ be a minimum join of $(G, T)$.

Theorem~\ref{thm:construct} implies that $A$ is a maximal bipartitic extreme set of $(G, T)$ with $A \cup D_A = V(G)$. 
Theorem~\ref{thm:decompose} further implies that $A$ is an $F$-combic set of $(G, T)$. 
For each $v\in B$, let $f_v$ be the edge from $\parcut{G}{G_v} \cap F$, and let $r_v$ be the end of $f_v$ from $V(G_v)$. 
Lemma~\ref{lem:combic2min} implies that $F\cap E_G[A, D_A]$ forms a minimum join of $(G', T')$, 
and $F\cap E(G_v)$ is a minimum join of $(G_v, T_v\Delta \{r_v\})$ for every $v\in B$.

\begin{pclaim} 
For every $x\in A$, $\distgtf{G}{T}{F}{r}{x} = 0$. 
For every $v\in B$ and every $x\in V(G_v)$, $\distgtf{G}{T}{F}{r}{x} < 0$. 
\end{pclaim} 
\begin{proof} 
Let $x\in V(G)$, and define a vertex $x'$ of $G'$ as follows: 
Let $x':= x$ if $x$ is a vertex in $A$; 
otherwise, let $x':= v$ where $v$ is the vertex from $B$ with $x\in V(G_v)$. 
Let $P$ be an $F\cap E_G[A, D_A]$-shortest path of $(G', T')$ between $r$ and $x'$. 

For each $w\in B \cap V(P)$, 
define $s_w \in V(G_w)$ as follows:   
If $w \neq x'$, then let $e_w$ be the edge from $\parcut{P}{w}\setminus \{f_w\}$, 
and let $s_w \in V(G_w)$ be the end of $e_w$ as an edge of $G$; 
in contrast, let $s_w := x$ if $w = x'$. 
Then, for each $w\in B \cap V(P)$,   
let $Q_w$ be an $F\cap E(G_w)$-shortest path of $G_w$ between $r_w$ and $s_w$. 

Let $R$ be the sum of $G. E(P)$ and  the paths $ Q_w$ where $w$ is taken over every  $w\in B \cap V(P)$. 
 Then, $R$ is a path between $r$ and $x$.
 The $F$-weight of $G. E(P)$ is equal to the $F\cap E_G[A, D_A]$-weight of $P$. 
Therefore, according to Lemma~\ref{lem:qcomb2dist}, it is equal to $0$ if $x\in A$ holds; otherwise, it is $-1$. 
For each $w\in B\cap V(P)$,  we have $w_F(Q_w) = 0$ if $w \neq x'$ holds; 
if $w = x'$ holds, we have $w_F(Q_w) \le 0$. 
Hence, $w_F(R) = 0$ for $x\in A$, whereas $w_F(R) < 0$ for $x\in V(G)\setminus A$. 
This proves the claim. 
\end{proof}

That is, $G$ is primal with respect to $r$ for which $A = \agtr{G}{T}{r}$. 
The theorem is proved. 
\end{proof}

Theorem~\ref{thm:primal2char} means 
that primal bipartite grafts can be characterized in the following constructive characterization form.

\begin{definition} 
For every vertex symbol $r$, let $\skcomb{r}$ denote the family of combs of the form $(G, T; A, B)$ 
such that $(G, T; A, B)$ is a primal with respect to  $r \in A$.  
Let $\primal$ denote the family of triples defined as follows. 
\begin{rmenum} 
\item For every vertex symbol $r$, 
if $(G, T; A, B)$ is a member of $\skcomb{r}$, then $(G, T; r)$ is a member of $\primal$. 
\item  
Let $(G, T; A, B) \in \skcomb{r}$, and let $(G_v, T_v; r_v)\in \primal$ for every $v\in B$. 
If $(\hat{G}, \hat{T})$ is a synthesis of  skeleton comb $(G, T; A, B)$ and  set of tooth grafts $\{(G_v, T_v; r_v): v\in B\}$, 
then $(\hat{G}, \hat{T}; r)$ is a member of $\primal$. 
\end{rmenum} 
\end{definition}

The family $\primal$ is well-defined by Theorem~\ref{thm:primal2char}. 
The next theorem is a restatement of Theorem~\ref{thm:primal2char}. 

\begin{theorem} 
For every vertex symbol $r$, 
$(G, T; r)\in \primal$ if and only if $(G, T)$ is a primal bipartite graft with respect to $r$. 
\end{theorem}

\section{Characterization of Primal Combs} 

Preceding sections reveal that primal bipartite grafts are important components of general bipartite grafts, 
and the last section reveals that primal combs are important components of primal bipartite grafts. 
In this section, we characterize primal combs in terms of the cathedral decomposition~\cite{kita2021bipartite}.

\begin{theorem}  \label{thm:nonposi2poset} 
A comb $(G, T; A, B)$ is primal with respect to $r\in A$ if and only if 
the poset $(\tcomp{G}{T}, \preceq)$ has the minimum element $C$ with $r\in A\cap V(C)$. 
\end{theorem} 
\begin{proof} 
We first prove the sufficiency. 
Assume that $(G, T; A, B)$ is a primal comb with respect to  $r\in A$. Let $F$ be a minimum join of $(G, T; A, B)$. 
Let $G_r \in \tcomp{G}{T}$ be the factor-component with $r\in V(G_r)$. 

\begin{pclaim} \label{claim:noedge} 
$E_G[ A\cap V(G_r), B\setminus V(G_r)] = \emptyset$. 
\end{pclaim} 
\begin{proof} 
Suppose, to the contrary, that $x\in A\cap V(G_r)$ and $y\in B\setminus V(G_r)$ are adjacent vertices. 
Let $P$ be an $F$-shortest path between $y$ and $r$;  Lemma~\ref{lem:qcomb2dist} implies $w_F(P) = -1$. 
Trace $P$ from $y$, and let $z$ be the first encountered vertex in $V(G_r)$. 
If $z\in A$ holds, then Lemma~\ref{lem:qcomb2path} implies $w_F(yPz) = -1$. 
Under Lemma~\ref{lem:elemcomb2nonposi}, let $Q$ be a path of $G_r$ between $x$ and $z$ with $w_F(Q) = 0$. 
Then, $yPz + Q + xy$ is a circuit of $F$-weight $0$ that contains non-allowed edges from $\parcut{G}{G_r}$, 
which contradicts Lemma~\ref{lem:circuit}. 
If $z\in B$ holds, then $w_F(yPz) = 0$. 
We can also find a path of $G_r$ between $x$ and $z$ with $F$-weight $-1$,  
and a similar discussion leads to a contradiction. 
This proves the claim. 
\end{proof} 

\begin{pclaim} \label{claim:dist} 
$(G, T; A, B)/G_r$ is a bipartite graft, and $F \setminus E(G_r)$ is a minimum join of this graft. 
 Among the paths in $(G, T; A, B)/G_r$ between $x\in V(G)\setminus V(G_r)$ and $[G_r]$,  
 every $F \setminus E(G_r)$-shortest one is of weight $0$ or $1$ if $x$ is in $A$ or $B$, respectively. 
\end{pclaim} 
\begin{proof} 
It is obvious from Claim~\ref{claim:noedge} that $(G, T; A, B)/G_r$ is a bipartite graft. 
It is also obvious that $F\setminus E(G_r)$ is a join of this graft. 
Furthermore, Lemma~\ref{lem:qcomb2dist} implies that the $F$-distance between any neighbors of the contracted vertex $[G_r]$ is $0$. 
Therefore,  $(G, T; A, B)/G_r$ has no circuit $C$ of negative $F\setminus E(G_r)$-weight.  
Hence, Lemma~\ref{lem:minimumjoin} proves that  $F\setminus E(G_r)$ is a minimum join of $(G, T; A, B)/G_r$.

Next, let $P$ be an $F$-shortest path of $(G, T)$ between $r$ and $x$. 
Trace $P$ from $x$, and let $z$ be the first encountered vertex in $G_r$. 
Claim~\ref{claim:noedge} implies that $w_F(xPz)$ is equal to $1$ or $0$ if $x$ is in $A$ or $B$, respectively.   
This path $xPz$ forms a  path of  $(G, T; A, B)/G_r$ between $x$ and $[G_r]$.  
Lemma~\ref{lem:qcomb2dist} ensures that it is $F\setminus E(G_r)$-shortest. 
Hence, we obtain the second claim. 
\end{proof} 

Claim~\ref{claim:dist} now proves that $(G, T; A, B)/G_r$ is a quasicomb with root $[G_r]$. 
This completes the proof of the sufficiency. 
The necessity can easily be proved by considering the concatenation of paths. 
\end{proof}

\begin{ac} 
This study is supported by JSPS KAKENHI Grant Number 18K13451. 
\end{ac}

\bibliographystyle{splncs03.bst}
\bibliography{tbicolor.bib}

\begin{thebibliography}{1}
\providecommand{\url}[1]{\texttt{#1}}
\providecommand{\urlprefix}{URL }

\bibitem{kita2020bipartite}
Kita, N.: Bipartite graft {I}: {D}ulmage-{M}endelsohn decomposition for combs.
  arXiv preprint arXiv:2007.12943  (2020)

\bibitem{kita2021bipartite}
Kita, N.: Bipartite graft {II}: Cathedral decomposition for combs. arXiv
  preprint arXiv:2101.06678  (2021)

\bibitem{DBLP:journals/jct/Sebo90}
Seb{\"{o}}, A.: Undirected distances and the postman-structure of graphs. J.
  Comb. Theory, Ser. {B}  49(1),  10--39 (1990)

\end{thebibliography}

\setcounter{theorem}{1}
\renewcommand{\thetheorem}{A.\arabic{theorem}}

\section*{Appendix}

The appendix provides some preliminary lemmas inherited from Kita~\cite{kita2020bipartite}. 

\begin{lemma}[see also Seb\"o~\cite{DBLP:journals/jct/Sebo90}] \label{lem:minimumjoin} 
Let $(G, T)$ be a graft, and let $F$ be a join of $(G, T)$. 
Then, $F$ is a minimum join of $(G, T)$ if and only if 
there is no circuit $C$ with $w_F(C) < 0$. 
\end{lemma}

\begin{lemma}[see also Seb\"o~\cite{DBLP:journals/jct/Sebo90}] \label{lem:circuit}  
Let $(G, T)$ be graft, and let $F$ be a minimum join. 
If $C$ is a circuit with $w_F(C) = 0$, then $F\Delta E(C)$ is also a minimum join of $(G, T)$. 
Accordingly, every edge of $C$ is allowed. 
\end{lemma}

\begin{lemma}[Kita~\cite{kita2020bipartite}] \label{lem:elemcomb2nonposi} 
If $(G, T; A, B)$ is a factor-connected comb, then 
\begin{rmenum} 
\item $\distgt{G}{T}{x}{y} = 0$ for every $x, y\in A$,  
\item $\distgt{G}{T}{x}{y} = -1$ for every $x\in A$ and every $y\in B$, and 
\item $\distgt{G}{T}{x}{y} \in \{ 0, -2\}$ for every $x, y \in B$. 
\end{rmenum} 
\end{lemma}

\begin{lemma}[Kita~\cite{kita2020bipartite}] \label{lem:qcomb2dist} 
Let $(G, T; A, B)$ be a quasicomb, and let $F$ be a minimum join of $(G, T; A, B)$. 
Then, the following properties hold. 
\begin{rmenum} 
\item $\distgtf{G}{T}{F}{x}{y} \ge 0$ for every $x, y \in A$. 
\item $\distgtf{G}{T}{F}{x}{y} \ge -1$ for every $x\in A$ and every $y\in B$. 
\item $\distgtf{G}{T}{F}{x}{y} \ge -2$ for every $x, y\in B$.  
\end{rmenum} 
\end{lemma}

\begin{lemma}[Kita~\cite{kita2020bipartite}] \label{lem:qcomb2path} 
Let $(G, T; A, B)$ be a quasicomb, and let $F$ be a minimum join of $(G, T; A, B)$.  
Let $x, y\in V(G)$, and let $P$ be a path between $x$ and $y$.   
\begin{rmenum} 
\item Let $x, y\in A$. Then, $w_F(P ) = 0$ holds if and only if $P$ is $F$-balanced.  
\item Let $x\in A$ and $y\in B$.  
 $w_F(P ) = -1$ holds if and only if $P$ is $F$-balanced and the edge of $P$ connected to $y$ is in $F$. 
\item Assume that $P$ is $F$-balanced. 
Then, $w_F(P) = 1$ holds if and only if the edge of $P$ connected to $y$ is not in $F$. 
\item Let $x, y\in B$, and assume that $P$ is $F$-balanced. 
Then, $w_F(P) = 0$ holds if and only if, of the two edges of $P$ connected to the ends,  one is in $F$ whereas the other is not in $F$. 
\end{rmenum} 
\end{lemma}

\end{document}